\documentclass[12pt,reqno]{amsart}
\usepackage{graphicx}
\usepackage{amssymb,amsmath}
\usepackage{amsthm}
\usepackage{color,graphicx}
\usepackage{hyperref}
\usepackage{color}
\usepackage{epstopdf}

\newcommand{\be}{\begin{equation}}

\newcommand{\ee}{\end{equation}}
\newtheorem{teo}{Theorem}[section]
\newtheorem{lema}{Lemma}[section]
\newtheorem{prop}{Proposition}[section]
\newtheorem{defi}{Definition}[section]
\newtheorem{obs}{Remark}[section]
\newtheorem{coro}{Corollary}[section]

\newcommand{\R}{{\mathbb R}}

\numberwithin{equation}{section}
\numberwithin{figure}{section}

\newtheorem{theorem}{Theorem}[section]
\newtheorem{proposition}[theorem]{Proposition}

\begin{document}
\vglue-1cm \hskip1cm
\title[Periodic Waves for a Dispersive Equation]{Orbital Stability of Periodic Traveling-Wave Solutions for a Dispersive Equation}

\begin{center}



\maketitle

{\bf F\'abio Natali}

{Departamento de Matem\'atica - Universidade Estadual de Maring\'a\\
Avenida Colombo, 5790, CEP 87020-900, Maring\'a, PR, Brazil.}\\
{ fmnatali@hotmail.com}

\end{center}

\begin{abstract}
In this paper we establish the
orbital stability of periodic traveling waves for a general class of
dispersive equations. We use the Implicit Function Theorem to
guarantee the existence of smooth solutions depending of the
corresponding wave speed. Essentially, our method establishes that if the linearized operator
has only one negative eigenvalue which is simple and zero is a simple eigenvalue the orbital stability is determined provided that
a convenient condition about the average of the wave is satisfied. We use our approach to prove the orbital stability of periodic dnoidal waves associated with the Kawahara equation.

\end{abstract}

\section{Introduction.}

The existence of solutions that
maintains its shape while it travels at constant speed is one of the most fascinating  phenomena determined by
dispersive equations. These special solutions (in general, called traveling waves) arise because of the perfect balance between the nonlinear and dispersive
effects in the medium. In current literature, the existence of these
solutions appear in several applications as fluid dynamics, nonlinear
optics, hydrodynamic and many other fields. Thus, it is important to establish a qualitative study of the dynamic related to these special solutions.\\
\indent The goal in this paper is to present sufficient conditions for the orbital stability of periodic
traveling wave solutions related to the following general dispersive model,
\begin{equation}
u_t+uu_x-(\mathcal{M}u)_x=0, \label{equakawa}
\end{equation}
where $u:\mathbb{R}\times\mathbb{R}\rightarrow\mathbb{R}$ is a real
$L-$periodic function and $\mathcal{M}$ is a differential or
pseudo-differential operator in periodic setting and it is defined as a Fourier multiplier operator by
\begin{equation}\label{symbol}
\widehat{\mathcal{M}g}(\kappa)=\theta
(\kappa)\widehat{g}(\kappa),\;\;\;\kappa\in \mathbb Z,
\end{equation}
where the symbol $\theta$ of $\mathcal{M}$ is assumed to be a
mensurable, locally bounded function on $\mathbb R$, satisfying
\begin{equation}\label{alpha}
 A_1|\kappa|^{m_2}\leq \theta(\kappa)\leq A_2 |\kappa|^{m_2},\ \ \ \ \ m_2>0,
\end{equation} for all
$\kappa\in \mathbb Z$ and for some $A_i>0$, $i=1,2$. Hypothesis $(\ref{alpha})$ is necessary to study qualitative aspects of the model $(\ref{equakawa})$ (for instance, global well-posedness and stability) in the respective energy space associated, namely, $H_{per}^{\frac{m_2}{2}}([0,L])$. Now, since $\theta(0)=0$ one has that $\mathcal{M}$ satisfies
\begin{equation}\label{cond}\mathcal{M}(a+u)=\mathcal{M}u\ \ \ \ \ \ \ \mbox{and}\ \ \ \ \
\ \ \int_0^L(\mathcal{M}u)dx=0,\ \ \ \ \ \forall\
a\in\mathbb{R}.\end{equation}

In equation $(\ref{equakawa})$, we consider traveling wave
solutions of the form $u(x,t)=\psi(x-\omega t)$, where $\omega\in
I\subset \mathbb R$ and $\psi:\mathbb R\to \mathbb R$ is a smooth
function. So, if we substitute this form  into $(\ref{equakawa})$,
we obtain after integration
\begin{equation}\label{soltrav11}
-\omega\psi_{(\omega,A)}+\frac{1}{2}\psi_{(\omega,A)}^{2}-\mathcal{M}\psi_{(\omega,A)}+A=0,
\end{equation}
where $A$ is a constant of integration not necessarily  zero. A
crucial role in our stability analysis is given by the symmetries of the model
(\ref{equakawa}) in $\mathbb R$, namely,
\begin{enumerate}
\item {\it translation invariance}: $u(x,t)\to u(x+y,t), \;y\in \mathbb R$;

\item  {\it Galilean invariance}: $u(x,t)\to a+u(x,t),\;a\in \mathbb R$.
\end{enumerate}
So, if one considers  the first  condition in (\ref{cond}) and the
Galilean invariance, we may assume   $A\equiv0$ in (\ref{soltrav11}) for a specific value of parameter $a$. In addition, the Galilean invariance can be also used to construct positive, negative and sign-changed periodic solutions by taking a convenient value of $a$.\\
\indent Particular cases of the operator $\mathcal{M}$ and the respective result of orbital stability of periodic waves have been obtained by an extensive number of contributors. For instance, if one considers
$\mathcal{M}=-\partial_x^2$  (the Korteweg-de Vries equation) we can cite
\cite{AN} \cite{ABS}, \cite{Haragus},
\cite{johnson}, and for $\mathcal{M}=\mathcal{H}\partial_x$ (the
Benjamin-Ono equation), where $\mathcal{H}$ indicates the Hilbert
transform in periodic context, the first result of orbital stability of periodic waves was treated in \cite{AN}. When $\mathcal{M}$ represents a fractionary derivative as $\mathcal{M}=(\sqrt{-\partial_x^2})^{\alpha}$, $0<\alpha\leq2$, in the Fourier sense (which includes the cases $\mathcal{M}=-\partial_x^2$ and $\mathcal{M}=\mathcal{H}\partial_x$), we have the work \cite{hur} where the authors assumed the existence of minimizers for the energy functional associated and proving the stability of periodic waves provided the number of negatives eigenvalues is one or two (to obtain the spectral property, they have used the approach in \cite{FL}).\\
\indent Next, we shall give a brief outline of our work. In fact, let us consider the linearized operator around the wave $\psi_{(\omega,A)}$
\begin{equation}\label{operator}
\mathcal{L}=\mathcal{M}+\omega-\psi_{(\omega,A)}.
\end{equation}
Operator $\mathcal{L}$ in $(\ref{operator})$ is a closed, unbounded,
self-adjoint operator on $L_{per}^2([0,L_0])$ whose spectrum consists in an
enumerable (infinite) set of eigenvalues. Thus, by assuming that $\mathcal{L}_{0}:=\mathcal{L}\big|_{(\omega,A)=(\omega_0,A_0)}$ has only one negative eigenvalue which is simple and zero is a simple eigenvalue whose associated eigenfunction is $\frac{d}{dx}\psi$ (as required in \cite{be}, \cite{bona1}, \cite{grillakis1} and \cite{weinstein1}), we are enable to establish the orbital stability of the periodic wave $\psi$ provided that the average of the wave satisfies $\frac{1}{L_0}\int_0^{L_0}\psi(x)dx>\omega_0$. Our approach will based on a combination of techniques determined by \cite{bona1}, \cite{grillakis1}, \cite{johnson} and \cite{weinstein1} where the construction of a smooth surface
$$
(\omega,A)\in \mathcal{O}\mapsto\psi_{(\omega,A)}\in H_{per,e}^n([0,L_0]),\ \
n\in \mathbb N,
$$
of periodic waves which solves equation $(\ref{soltrav11})$ is relevant in our analysis. Thus, in order to summarize our main assumption, we highlight it as follows \\

\begin{itemize}
\item[$(H)$] Let $(\omega_0,A_0)\in\mathbb R_{+}\backslash\{0\}\times\mathbb{R}$ be fixed. Suppose that
$\psi:=\psi_{(\omega_0,A_0)}\in C_{per}^{\infty}([0,L_0])$ is a positive even periodic
traveling wave solution for the equation $(\ref{soltrav11})$ with fixed period $L_0>0$.
Moreover, the self-adjoint operator
$\mathcal{L}_{0}$
 has only one negative eigenvalue which is
simple and zero is a simple eigenvalue
whose eigenfunction is $\frac{d}{dx}\psi$.
\end{itemize}
\vspace{0.2cm}

\indent As an application of our work, we present the result of orbital stability of periodic traveling waves for the Kawahara equation
\begin{equation}\label{equakawa1}
u_t+uu_x+u_{xxx}-u_{xxxxx}=0,
\end{equation}
that is, $\mathcal{M}=\partial_x^4-\partial_x^2$ in equation
$(\ref{equakawa})$. The existence of explicit solutions is determined by using exhaustive numerical computations. In \cite{parkes}, the authors put forward an explicit periodic wave having a \textit{dnoidal} profile as
\begin{equation}\label{dnprof}\begin{array}{lll}
\psi(x)=a &+& b\left(\mbox{dn}^2\left(\frac{2K}{L}x,k\right)-\frac{E}{K}\right)\\\\
&+& d\left(\mbox{dn}^4\left(\frac{2K}{L}x,k\right)-(2-k^2)\frac{2E}{3K}+\frac{1-k^2}{3}\right),\end{array}
\end{equation}
where \textit{dn} is the Jacobi elliptic function called dnoidal, $k\in(0,1)$ is the modulus, $K=K(k)$ indicates the complete integral elliptic of first kind and parameters $a,\ b$ and $d$ depend smoothly on the modulus $k\in(0,1)$. Regarding the stability, in \cite{haragus1} the authors showed the linear stability of periodic waves (that is, the spectrum of the linearization about these waves is contained in the imaginary axis) related to the equation \eqref{equakawa1}. They established the periodic travelling waves with speed $\omega$ are spectrally stable provided that the amplitude $a$ of the wave satisfies $a= o(|\omega|^{5/4})$. In \cite{ncp} it was determined a local proof for the orbital stability of periodic waves having the form $(\ref{dnprof})$ by using the arguments in \cite{andrade}. Our goal is to determine a more complete scenario for the stability of periodic waves.\\
\indent Our paper is organized as follows. Section 2
is devoted to present the stability of periodic waves associated
with the general equation $(\ref{equakawa})$.
In Section 3 we present the application of the results in
previous section.

\section{Stability of Periodic Waves}

\indent Before starting, we need to guarantee the existence of a smooth surface of periodic waves having fixed period. We see that assumption $(H)$ is sufficient for our purpose.

\begin{teo}\label{teoexist} Let us suppose that assumption $(H)$ holds. There is a smooth surface of positive even periodic solutions for $(\ref{soltrav11})$ and an open subset $\mathcal{O}\subset\mathbb R_{+}\backslash\{0\}\times\mathbb{R}$, containing $(\omega_0,A_0)$, such that
$$
(\omega,A)\in \mathcal{O}\mapsto\psi_{(\omega,A)}\in H_{per,e}^n([0,L_0]),\ \
n\in \mathbb N,
$$
all of them with the same minimal period $L_0>0$.
\end{teo}

\begin{proof}
We define for $s\geq 0$,  $X_{e}^s=\{f\in H_{per}^s([0,L]): f
\;\text{is even}\}$. Let $\Pi:R_{+}\backslash\{0\}\times\mathbb{R}\times X_{e}^{m_2}\rightarrow X_{e}^0$ be the map defined
by
\begin{equation}\label{operaA}
\Pi(\omega,A,\psi)=
\mathcal{M}\psi+\omega\psi-\frac{1}{2}\psi^2+A.
\end{equation}
Function $\Pi$ is smooth in all variables and from assumption $(H)$ one has
$\Upsilon(\omega_0,A_0,\psi)=0$. Next, the
Fr\'echet derivative associated with the function $\Pi$ with
respect to $\psi$ evaluated at the point
$(\omega_0,A_0,\psi)$ becomes an operator
$\mathcal{G}$ given by

\begin{equation}\label{operaG}
\mathcal{G}=\mathcal{M}+\omega_0-\psi
\end{equation}
Now, let us consider $f\in \ker(\mathcal{G})$, where
$\mathcal{G}$ is defined on $X_{e}^0$ with domain
$D(\mathcal{G})= X_{e}^{m_2}$. So, we have
$$
\mathcal{M}f+\omega_0f-\psi f=0
$$
Then $\frac{d}{dx}\psi$ is an eigenfunction of the operator
$\mathcal{G}:=\mathcal{M}+\omega_0-\psi$
(as an operator defined in $X^0$ with domain $X^{m_2}$) whose
eigenvalue is $\lambda=0$. Moreover, since $\frac{d}{dx}\psi$
is odd and it does not belong to $X_{e}^{m_2}$, we see that $\mathcal{G}$ is one to one. Now, let us prove
that, with domain $X_{e}^{m_2}$, $\mathcal{G}$ is also surjective.
Indeed, $\mathcal{G}$ is clearly a self-adjoint operator. Thus
$\sigma(\mathcal{G})=\sigma_{disc}(\mathcal{G})\cup
\sigma_{ess}(\mathcal{G})$. Since $X_{e}^{m_2}$ is compactly
embedded in $X_{e}^0$, the operator $\mathcal{G}$ has compact
resolvent. Consequently, $\sigma_{ess}(\mathcal{G})=\emptyset$ and
$\sigma(\mathcal{G})=\sigma_{disc}(\mathcal{G})$ consists of
isolated eigenvalues with finite algebraic multiplicities (see
\cite{kato1}). Now, since $\mathcal{G}$ is one-to-one, it follows
that 0 is not an eigenvalue of $\mathcal{G}$, and so it does not
belong to $\sigma(\mathcal{G})$. This means that $0\in
\rho(\mathcal{G})$, where $\rho(\mathcal{G})$  denotes the resolvent
set of $\mathcal{G}$, and so, by definition, $\mathcal{G}$ is
surjective. The arguments above imply  that $\mathcal{G}^{-1}$
exists and, moreover, is a bounded linear operator. Consequently,
since $\Pi$ and $\Pi_{\psi}$ are clearly smooth maps on
their domains, from the Implicit Function Theorem we establish the
results enunciated above.
\end{proof}
 \indent Next result establishes the behaviour of the first eigenvalues associated with the linearized operator $\mathcal{L}_0$ in $(\ref{operator})$.

\begin{prop}\label{teo5}
Suppose that assumption $(H)$ holds and let $\psi_{(\omega,A)}$ be the periodic
traveling wave solution obtained in Theorem $\ref{teoexist}$. There exists an open neighbourhood $\widetilde{\mathcal{O}}\subset\mathcal{O}$ containing $(\omega_0,A_0)$ such that the linearized operator
$\mathcal{L}=\mathcal{M}+\omega-\psi_{(\omega,A)}$, $(\omega,A)\in\widetilde{\mathcal{O}}$,
has only one negative eigenvalue which is simple and zero is a
simple eigenvalue whose eigenfunction is $\frac{d}{dx}\psi_{(\omega,A)}$.
\end{prop}
\begin{proof}
Indeed, from Theorem $\ref{teoexist}$ let us consider $\mathcal{O}$ the 
open neighbourhood containing $(\omega_0,A_0)$. Choose a convenient open neighbouhood $\widetilde{\mathcal{O}}\subset\mathcal{O}$ containing $(\omega_0,A_0)$ (for instance, an open ball centered at $(\omega_0,A_0)$ with  sufficiently small radius). The family of self-adjoint
operators $\mathcal{L}=\mathcal{M}+\omega-\psi_{(\omega,A)}$ is defined on
$L_{per}^2([0,L])$ with domain
$D(\mathcal{L})=H_{per}^{m_2}([0,L])$. In what follows, we consider the
metric {\it gap}, $\widehat{\delta}(T,S)$, between the closed
operators $T$ and $S$ (see Chap. IV in \cite {kato1}). From
Theorem 2.17 and Theorem 2.14 in Chap. IV of \cite {kato1},
\begin{equation}
\begin{aligned}
\widehat{\delta}(\mathcal{L}_{(\omega_0,A_0)},\mathcal{L})&\leq
2(1+\|\psi_{(\omega,A)}\|_{L^\infty}^2)\widehat{\delta}
(\mathcal{L}_{(\omega_0,A_0)}+\psi_{(\omega,A)},\mathcal{M} +{\omega})\\
&\leq 2(1+\|\psi_{(\omega,A)}\|_{L^\infty}^2)
[|\omega_0-\omega|+\|\psi_{(\omega,A)}-\psi_{(\omega_0,A_0)}\|_{L^\infty}].
\end{aligned}
\end{equation}
Therefore we obtain
$\widehat{\delta}(\mathcal{L}_{(\omega_0,A_0)},\mathcal{L})\to 0$ as $(\omega,A)\to
(\omega_0,A_0)$, and so from \cite[Theorem 3.16, Chap. IV]{kato1}) the isolated
eigenvalues of $\mathcal{L}_{(\omega_0,A_0)}$ are stable. Hence, for $(\omega,A)\in\widetilde{\mathcal{O}}$, we obtain that
$\mathcal{L}$  has the same spectral properties of
$\mathcal{L}_{(\omega_0,A_0)}$.
\end{proof}

\indent Next, we present our stability result by adapting the arguments in \cite{bona2}, \cite{grillakis1}, \cite{johnson} and \cite{weinstein1}. So, in what follows, we assume that the model in $(\ref{equakawa})$ possesses a convenient global well-posedness result in the space $H_{per}^{s}([0,L_0])$,
for $s\geq \frac{m_2}{2}$. In addition, we need to suppose the existence of the following conserved quantities
\begin{equation}\label{conser1}
E(u)=\frac{1}{2}\int_0^{L_0}(\mathcal{M}^{1/2}u)^2-\frac{1}{3}u^3dx,
\end{equation}
\begin{equation}\label{conser2}
F(u)=\frac{1}{2}\int_0^{L_0}u^2dx,
\end{equation}
and
\begin{equation}\label{conser3}
M(u)=\int_0^{L_0}udx,
\end{equation}
where in the quantity $(\ref{conser1})$ we are using that operator $\mathcal{M}$ is $m-accretive$ (see \cite[pg. 281]{kato1}). This fact allows us to conclude the existence of a self-adjoint linear operator $\mathcal{M}^{1/2}$ such that $(\mathcal{M}^{1/2})^2=\mathcal{M}$.\\
\indent Assume that assumption $(H)$ holds. From Theorem $\ref{teoexist}$ we are enabled to consider
$$
\eta:=\frac{\partial}{\partial\omega}\psi_{(\omega,A)}\Big|_{(\omega,A)=(\omega_0,A_0)},\ \qquad\beta:=\frac{\partial}{\partial A}\psi_{(\omega,A)}\Big|_{(\omega,A)=(\omega_0,A_0)}.
$$
Define
 $$
 M(\psi)=\int_0^{L_0}\psi_{(\omega,A)}(x)dx\Big|_{(\omega,A)=(\omega_0,A_0)},\qquad F(\psi)=\frac{1}{2}\int_0^{L_0}\psi_{(\omega,A)}^2(x)dx\Big|_{(\omega,A)=(\omega_0,A_0)},$$
$$ M_{\omega}(\psi)=\frac{\partial}{\partial\omega}\int_0^{L_0}\psi_{(\omega,A)}(x)dx\Big|_{(\omega,A)=(\omega_0,A_0)},\qquad  M_{A}(\psi)=\frac{\partial}{\partial A}\int_0^{L_0}\psi_{(\omega,A)}(x)dx\Big|_{(\omega,A)=(\omega_0,A_0)},
 $$
 and
 $$
 F_{\omega}(\psi)=\frac{1}{2}\frac{\partial}{\partial{\omega}}\int_0^{L_0}\psi_{(\omega,A)}^2(x)dx\Big|_{(\omega,A)=(\omega_0,A_0)}, \qquad F_{A}(\psi)=\frac{1}{2}\frac{\partial}{\partial{A}}\int_0^{L_0}\psi_{(\omega,A)}^2(x)dx\Big|_{(\omega,A)=(\omega_0,A_0)}.
 $$
In order to simplify the notation, the  norm and inner product in $L_{per}^2([0,L_0])$ will be denoted by  $||\cdot||$ and $\langle\cdot,\cdot\rangle$.

\indent Now, we need some preliminaries notations. Let $\rho$ be the semi-distance defined on the energy space $X^{\frac{m_2}{2}}=H_{per}^{\frac{m_2}{2}}([0,L_0])$ as
\be\label{rho}
\rho(u,\psi)=\inf_{y\in\mathbb{R}}||u(\cdot+y)-\psi||_{X^{\frac{m_2}{2}}}.
\ee
For a given $\varepsilon>0$, we define the $\varepsilon$-neighborhood of the orbit $O_\psi$ as
\be \label{tube}U_{\varepsilon} := \{u\in X^{\frac{m_2}{2}};\ \rho(u,\psi) < \varepsilon\}.\ee
We also introduce the smooth manifolds
\be\label{manifold}
\Sigma_0=\{u\in X^{\frac{m_2}{2}};\ F(u)=F(\psi),\ M(u)=M(\psi)\},
\ee
and
\be\label{tau0}
\Upsilon_0=\{u\in X^{\frac{m_2}{2}};\ \langle \psi,u\rangle=\langle 1,u\rangle=0\}.
\ee
\indent Our notion of orbital stability is finally presented.
\begin{defi}\label{defi1}
We say that $\psi$ is orbitally stable with respect to (\ref{equakawa}) if, for all $\varepsilon>0$, there exists $\delta>0$ such that if $\|u_0-\psi\|_{X^{\frac{m_2}{2}}}<\delta$ and $u(t)$ is the solution of (\ref{equakawa}) with $u(0)=u_0$, then
$$\rho(u(t),\psi)<\varepsilon, \quad \mbox{for all} \, \,  t\in\R.$$
\end{defi}
The next result state that under a suitable restriction, the operator $\mathcal{L}$ is strictly positive.

\begin{proposition}\label{prop2}
Suppose that assumption $(H)$ holds. Assume that there is $\Phi\in X^{m_2}$ such that $\langle\mathcal{L}_0\Phi,\varphi\rangle=0$, for all $\varphi\in \Upsilon_0$, and
\be\label{defiI}
\mathcal{I}:=\langle\mathcal{L}_0\Phi,\Phi\rangle<0
\ee
Then, there is a constant $c>0$ such that
$$\langle\mathcal{L}_0v,v\rangle\geq c||v||_{X^{\frac{m_2}{2}}}^2,$$
for all $v\in \Upsilon_0$ such that $\langle v,\psi'\rangle=0$.
\end{proposition}
\begin{proof}
We shall give only a sketch of the proof. From assumption $(H)$ one has
\be\label{decomp}L_{per}^2([0,L_0])=[\chi]\oplus [\psi']\oplus P,\ee
where $\chi$ satisfies $||\chi||=1$ and $\mathcal{L}_0\chi=-\lambda_0^2\chi$, $\lambda_0\neq0$.  By using the arguments  in \cite[page 278]{kato1}, we obtain that $$\langle \mathcal{L}_0p,p\rangle\geq c_1||p||^2,\ \ \ \ \ \mbox{for all}\ p\in D(\mathcal{L})\cap P,$$
where $c_1$ is a positive constant.

Next, from $(\ref{decomp})$, we write
$$\
\Phi=a_0\chi+b_0\psi'+p_0,\ \ \ \ \ a_0,b_0\in\mathbb{R},
$$
where $p_0\in D(\mathcal{L}_0)\cap P$. Now, since $\psi'\in \ker (\mathcal{L}_0)$, $\mathcal{L}_0\chi=-\lambda_0^2\chi$, and $\mathcal{I}<0$, we obtain
\be
\langle \mathcal{L}_0 p_0,p_0\rangle=\langle\mathcal{L}_0(\Phi-a_0\chi-b_0\psi'),\Phi-a_0\chi-b_0\psi'\rangle
=\langle\mathcal{L}_0\Phi,\Phi\rangle+a_0^2\lambda_0^2<a_0^2\lambda_0^2.
\ee
\indent Taking $\varphi\in \Upsilon_0$ such that $||\varphi||=1$ and $\langle \varphi,\psi'\rangle=0$, we can write $\varphi=a_1\chi+p_1$, where $p_1\in X^{\frac{m_2}{2}}\cap P$. Thus,
\begin{equation}
0=\langle\mathcal{L}_0\Phi,\varphi\rangle=\langle -a_0\lambda_0^2\chi +\mathcal{L}_0p_0,a_1\chi+p_1\rangle
=-a_0a_1\lambda_0^2+\langle\mathcal{L}_0p_0,p_1\rangle.
\end{equation}
The rest of the proof runs as in \cite[Lemma 5.1]{bona2} (see also \cite[Lemma 4.4]{johnson}).

\end{proof}

 Proposition \ref{prop2} is useful to establish  the following result.

\begin{proposition}\label{prop3}
Let $E$ be the conserved quantity defined in $(\ref{conser1})$. Under the assumptions of Proposition $\ref{prop2}$ there are $\alpha>0$ and $C=C(\alpha)>0$ such that
$$E(u)-E(\psi)\geq C\rho(u,\psi)^2,$$
for all $u\in U_{\alpha}\cap\Sigma_0$.
\end{proposition}
\begin{proof}
The proof can be found in \cite[Lemma 4.6]{johnson}. So, we omit the details.
\end{proof}

\indent Finally, we present sufficient conditions for the stability.

\begin{theorem}\label{teo2}
Assume that assumption $(H)$ holds and let us suppose that
$$
\mathcal{D}:=\left[\begin{array}{llll} F_{A}(\psi)\ \ M_A(\psi)\\
F_{\omega}(\psi)\ \ M_{\omega}(\psi)\end{array}\right]
$$ is invertible. If there is $\Phi\in X^{m_2}$ such that $\langle\mathcal{L}_0\Phi,\varphi\rangle=0$, for all $\varphi\in \Upsilon_0$, and $\mathcal{I}=\langle\mathcal{L}_0\Phi,\Phi\rangle<0$, then $\psi$ is orbitally stable in $X^{\frac{m_2}{2}}$ by the periodic flow of   $(\ref{equakawa})$.
\end{theorem}
\begin{proof}
Let $\alpha>0$ be the constant such that Proposition $\ref{prop3}$ holds.  Since $E$ is continuous at $\psi$, for a given $\varepsilon>0$, there exists $\delta\in (0,\alpha)$ such that if $\|u_0-\psi\|_{X^{\frac{m_2}{2}}}<\delta$  one has
\be\label{estepsilon}
E(u_0)-E(\psi)<M\varepsilon^2,
\ee
where $M>0$ is the constant in Proposition \ref{prop3}. We need to divide our proof into two cases.\\
\indent \textit{First case.} $u_0\in \Sigma_0$. Since $F$ and $M$ are conserved quantities, if $u_0\in\Sigma_0$ one has that $u(t)\in \Sigma_0$, for all $t\geq0$. The time continuity of the function $\rho(u(t),\psi)$ allows to choose $T>0$ such that \be\label{subalpha}\rho(u(t),\psi)<\alpha,\ \ \  \mbox{for all}\ t\in [0,T).\ee
Thus, one obtains $u(t)\in U_{\alpha}$, for all $t\in[0,T)$. Combining Proposition \ref{prop3} and $(\ref{estepsilon})$, we have
\be\label{estepsilon1}
\rho(u(t),\psi)<\varepsilon,\ \ \ \ \ \mbox{for all}\ t\in[0,T).
\ee
\indent Next, we prove that $\rho(u(t),\psi)<\alpha$, for all $t\in [0,+\infty)$, from which one concludes the orbital stability restricted to perturbations in the manifold $\Sigma_0$. Indeed, let  $T_1>0$ be the supremum of the values of $T>0$ for which $(\ref{subalpha})$ holds. To obtain a contradiction, suppose that $T_1<+\infty$.  By choosing $\varepsilon<\frac{\alpha}{2}$ we obtain, from $(\ref{estepsilon1})$,
$$
\rho(u(t),\psi)<\frac{\alpha}{2}, \ \ \ \ \ \mbox{for all}\ t\in[0,T_1).
$$
Since $t\in(0,+\infty)\mapsto\rho(u(t),\psi)$ is continuous, there is $T_0>0$ such that
$\rho(u(t),\psi)<\frac{3}{4}\alpha<\alpha$, for $t\in [0,T_1+T_0)$, contradicting the maximality of $T_1$. Therefore, $T_1=+\infty$ and the theorem  is established if $u_0\in\Sigma_0$.\\
\indent \textit{Second case.} $u_0\notin \Sigma_0$. In this case, since $\det(\mathcal{D})\neq 0$, we claim that there is  $(\omega_1,A_1)\in\mathcal{O},$  such that $F(\psi_{(\omega_1,A_1)})=F(u_0)$ and $M(\psi_{(\omega_1,A_1)})=M(u_0)$.\\
\indent In fact, since $M$ and $F$ are smooth, the Inverse Function Theorem implies the existence of $r_1,r_2>0$ such that the map
$$\begin{array}{ccc}\Gamma:B_{r_1}(\omega_0,A_0)&\longrightarrow& B_{r_2}(M(\psi),F(\psi))\\
(\omega,A)&\mapsto& (M(\psi_{(\omega,A)}),F(\psi_{(\omega,A)}))\end{array},$$

\noindent is a smooth diffeomorphism. Here, $B_{r}((x,y))$ denotes the  open ball in $\mathbb{R}^2$ centered in $(x,y)$ with radius $r>0$.  The continuity of the functionals $M$ and $V$ gives (if necessary we can take a smaller $\delta>0$)
$$|M(u_0)-M(\psi)|<\dfrac{r_2}{\sqrt2} \quad \mbox{and} \quad |F(u_0)-F(\psi)|<\dfrac{r_2}{\sqrt2},
$$
 that is, $(M(u_0),F(u_0))\in B_{r_2}(M(\psi),F(\psi))$. Since $\Gamma$ is a diffeomorphism, there is a unique $(\omega_1,A_1)\in B_{r_1}(\omega_0,A_0)$ such that $(M(u_0),F(u_0))=(M(\psi_{(\omega_1,A_1)}),F(\psi_{(\omega_1,A_1)}))$. The claim is thus proved.\\
\indent The remainder of the proof follows from the smoothness of the periodic wave with respect to the parameters, the fact that the period does not change whether $(\omega,A)\in\mathcal{O}$ and the triangle inequality.
\end{proof}

Theorem \ref{teo2} establishes the orbital stability of $\psi$ provided $det(\mathcal{D})\neq0$ and $\mathcal{I}<0$. The next proposition gives a sufficient condition to show that $\mathcal{I}<0$.

\begin{proposition}\label{propKpos}
Let $P:\R^2\to\R$ be the function defined as
$$
P(x,y)=x^2F_{\omega}(\psi)+xy(F_A(\psi)+M_{\omega}(\psi))+y^2M_A(\psi).
$$
Assume that there is $(x_0,y_0)\in\R^2$ such that $P(x_0,y_0)>0$. Then there is $\Phi\in X^{m_2}$ such that   $\langle\mathcal{L}_0\Phi,\varphi\rangle=0$, for all $\varphi\in \Upsilon_0$, and
$$
\mathcal{I}=\langle \mathcal{L}_0\Phi,\Phi \rangle<0.
$$
\end{proposition}
\begin{proof}
It suffices to define  $\Phi:=x_0\eta+y_0\beta$. Indeed, since $\mathcal{L}_0\beta=-1$ and $\mathcal{L}_0\eta=-\psi$, it is clear that  $\langle\mathcal{L}_0\Phi,\varphi\rangle=0$, for all $\varphi\in \Upsilon_0$, and
\[
\begin{split}
 \langle\mathcal{L}_0\Phi,\Phi \rangle&=\langle-x_0\psi-y_0,x_0\eta+y_0\beta\rangle\\
 &=-(x_0^2F_{\omega}(\psi)+x_0y_0F_A(\psi)+x_0y_0M_{\omega}(\psi)+y_0^2M_A(\psi))\\
 &=-P(x_0,y_0).
\end{split}
\]
The proof is thus completed.
\end{proof}

\begin{coro}\label{coro1}
Suppose that assumption $(H)$ occurs. If $\frac{M(\psi)}{L_0}>\omega_0>0$ then there exists $(x_0,y_0)\in\R^2$ such that $P(x_0,y_0)>0$.
\end{coro}
\begin{proof}
From assumption $(H)$ one gets Theorem $\ref{teoexist}$ and consequently, it is possible to derive equation $(\ref{soltrav11})$ with respect to $\omega$ and $A$ to get, respectively
\begin{equation}\label{traveta}
\mathcal{M}\eta+\omega\eta+\psi-\psi\eta=0,
\end{equation}
and,
\begin{equation}\label{travbeta}
\mathcal{M}\beta+\omega\beta-\psi\beta+1=0.
\end{equation}
Next, integrating equations $(\ref{traveta})$ and $(\ref{travbeta})$ over $[0,L_0]$ we deduce, respectively
\begin{equation}\label{Fomega}
F_{\omega}(\psi)=\omega M_{\omega}(\psi)+M(\psi),
\end{equation}
and,
\begin{equation}\label{FA}
F_{A}(\psi)=\omega M_{A}(\psi)+L_0.
\end{equation}

\indent On the other hand, since $1\in D(\mathcal{L})$ we have from $(\ref{operator})$ that $\mathcal{L}1=\omega-\psi_{(\omega,A)}$. The fact that $1,\psi_{(\omega,A)}\in [\psi_{(\omega,A)}']^{\bot}$, for all $(\omega,A)\in \mathcal{O}$, enables us to obtain $1=\omega\mathcal{L}^{-1}1-\mathcal{L}^{-1}\psi_{(\omega,A)}$, for all $(\omega,A)\in \mathcal{O}$. Therefore, from $(\ref{operator})$, $(\ref{traveta})$ and $(\ref{travbeta})$ we conclude
\begin{equation}\label{relFF}
L_0=-\omega_0 M_A(\psi)+M_{\omega}(\psi).
\end{equation}
So, collecting the results in $(\ref{Fomega})$, $(\ref{FA})$ and $(\ref{relFF})$ we have
\begin{equation}\label{Pab}\begin{array}{llll}
P(x_0,y_0)&=&x_0^2F_{\omega}(\psi)+x_0y_0F_A(\psi)+x_0y_0M_{\omega}(\psi)+y_0^2M_A(\psi)\\\\
&=&\left(x_0^2\omega_0+2x_0y_0+\frac{y_0^2}{\omega_0}\right)M_{\omega}(\psi)+x_0^2M(\psi)-\frac{y_0^2L_0}{\omega_0}
\end{array}\end{equation}
\indent Choosing $y_0\neq0$ and $x_0=-\frac{y_0}{\omega_0}$, one has from $(\ref{Pab})$
\begin{equation}\label{Pab1}
P(x_0,y_0)=\frac{y_0^2}{\omega_0^2}M(\psi)-\frac{y_0^2L_0}{\omega_0}=\frac{y_0^2L_0}{\omega_0^2}\left(\frac {M(\psi)}{L_0}-\omega_0\right).
\end{equation}
The fact that $\frac{M(\psi)}{L_0}>\omega_0$ enables us to finish the proof.\end{proof}

\begin{coro}\label{coro2}
Suppose that assumption $(H)$ occurs. Thus $$\det(\mathcal{D})=\frac{L_0}{\omega_0}\left(\omega_0-\frac{M(\psi)}{L_0}\right)M_{\omega}(\psi)+M(\psi)\frac{L_0}{\omega_0}.$$ In particular, if $\frac{M(\psi)}{L_0}>\omega_0>0$ and 
$M_{\omega}(\psi)<0$, we obtain that $\det(\mathcal{D})\neq0$ and the periodic wave $\psi$ is orbitally stable in the sense of Definition $\ref{defi1}$.
\end{coro}
\begin{proof}
In fact, from $(\ref{Fomega})$, $(\ref{FA})$ and $(\ref{relFF})$ we have
\begin{equation}\label{detD}\begin{array}{lllll}
\det(\mathcal{D})&=&F_A(\psi)M_{\omega}(\psi)-F_{\omega}(\psi)M_{A}(\psi)\\\\
&=&\omega_0M_A(\psi)M_{\omega}(\psi)+L_0M_{\omega}(\psi)-\omega_0M_{\omega}(\psi)M_A(\psi)-M(\psi)M_A(\psi)\\\\
&=&L_0M_{\omega}(\psi)-M(\psi)M_A(\psi)=L_0M_{\omega}(\psi)+\frac{M(\psi)L_0}{\omega_0}-\frac{M_{\omega}(\psi)M(\psi)}{\omega_0}\\\\
&=& \frac{L_0}{\omega_0}\left(\omega_0-\frac{M(\psi)}{L_0}\right)M_{\omega}(\psi)+\frac{M(\psi)L_0}{\omega_0}.
\end{array}
\end{equation}
From Proposition $\ref{propKpos}$ and Corollary $\ref{coro1}$, we deduce from Theorem $\ref{teo2}$ that the periodic wave $\psi$ is orbitally stable in the sense of the Definition $\ref{defi1}$.
\end{proof}

\indent Corollary $\ref{coro2}$ guarantees the orbital stability provided that assumption in $(H)$ is satisfied joint with 
$\frac{M(\psi)}{L_0}>\omega_0>0$  and $M_{\omega}(\psi)<0$. Thus, it remains to prove what happens with the orbital stability of $\psi$, if one considers the cases $\frac{M(\psi)}{L_0}>\omega_0>0$ and $M_{\omega}(\psi)\geq0$. This particular case is determined in a different way since we can not assure that $\det(\mathcal{D})\neq0$ in order to apply the arguments in Theorem $\ref{teo2}$. However, it is easy to see from $(\ref{Fomega})$ that $P(1,0)=F_{\omega}(\psi)=\omega_0M_{\omega}(\psi)+M(\psi)>0$. This information about the positivity of $F_{\omega}(\psi)$ enables us to enunciate the following result.\\
\begin{coro}\label{coroest1}
Suppose that assumption $(H)$ occurs. Let us assume that $\frac{M(\psi)}{L_0}>\omega_0>0$ and 
$M_{\omega}(\psi)\geq0$. Thus, the periodic wave $\psi$ is orbitally stable in the sense of Definition $\ref{defi1}$.
\end{coro}
\indent To prove Corollary $\ref{coroest1}$, we need to follow the arguments contained in \cite{natali1}. In fact, let us consider
\begin{equation}\label{eq21}\mathcal{P}_{(\omega,A)}=E+\omega F+AM\end{equation}
and the perturbation
\begin{equation}\label{eq22}u(x+y,t)=\psi_{(\omega,A)}(x)+v(x,t),\end{equation} where
$y=y(t)$ is the minimum point of the function
$$\Gamma_t(s)=\int_0^{L_0}(\mathcal{M}^{1/2}(u(x+y,t)-\psi_{(\omega,A)}))^{2}dx
+\omega\int_0^{L_0}(u(x+y,t)-\psi_{(\omega,A)}(x))^2dx,$$
$y\in\mathbb{R}$, and function $v$ satisfies the compatibility
condition
\begin{equation}\label{compat}\int_{0}^{L_0}\psi_{(\omega,A)}(x)\psi_{(\omega,A)}'(x)v(x,t)dx=0,\end{equation}
for all $t\in\mathbb{R}$.\\
\indent Thus, we obtain from $(\ref{eq21})$ and $(\ref{eq22})$ the following inequality
\begin{equation}\label{deltaR}\begin{array}{lllll}
\Delta\mathcal{P}_{(\omega,A)}&:=&\displaystyle\mathcal{P}_{(\omega,A)}(u)-\mathcal{P}_{(\omega,A)}(\psi_{(\omega,A)})
=\mathcal{P}_{(\omega,A)}(\psi_{(\omega,A)}+v)-\mathcal{P}_{(\omega,A)}(\psi_{(\omega,A)})\\\\
&\geq&\displaystyle
\frac{1}{2}\langle\mathcal{L}v,v\rangle-C_0||v||_{X^{\frac{m_2}{2}}}^3,\end{array}
\end{equation}
where $C_0\in\mathbb{R}$ is a positive constant which depend on the periodic wave $\psi_{(\omega,A)}$ and the constant of the Sobolev embeddings $X^{\frac{m_2}{2}}\hookrightarrow L_{per}^p([0,L_0])$, $p\geq2$, integer.\\
\indent Next, it is necessary to use the works due to
\cite{be} and \cite{bona1}, to establish convenient bounds
for the term $\langle\mathcal{L}v,v\rangle$. First, we need a preliminary result.
\begin{lema}\label{weinstein}
Let $\psi_{(\omega,A)}$ be as in Theorem $\ref{teoexist}$. Let
$\mathcal{L}$ be the self-adjoint operator defined in $(\ref{operator})$.
 We define
$$-\infty<w:=\min_{\phi}\{
\langle\mathcal{L}\phi,\phi\rangle; \;\;||\phi||_{L_{per}^2}=1\;\;
and \;\; \langle\phi,\psi_{(\omega,A)}\rangle=0 \}.
$$
Assuming that $\langle\chi_{(\omega,A)},\psi_{(\omega,A)}\rangle\neq0$
and $\psi_{(\omega,A)}\in [\ker(\mathcal{L})]^{\bot}$,
where $\chi_{(\omega,A)}$ is the eigenfunction associated with the
negative eigenvalue of $\mathcal{L}$. Then, if
\begin{equation}\label{weinsteincond}\langle\mathcal{L}^{-1}\psi_{(\omega,A)},\psi_{(\omega,A)}\rangle\leq0,\end{equation}
it follows that $w\geq 0$.
\end{lema}
\begin{proof}
See Lemma E.1 in \cite{weinstein2}.
\end{proof}
\begin{obs} To obtain that $\langle\chi_{(\omega,A)},\psi_{(\omega,A)}\rangle\neq0$ we need to use a Krein-Ruttman Theorem to guarantee that the eigenfunction $\chi_{(\omega,A)}$ related to the first eigenvalue of $\mathcal{L}$ is one-signed on $\mathbb{R}$ and the fact that $\psi_{(\omega,A)}$ is positive.
\end{obs}
Next, since
$$\langle\mathcal{L}^{-1}\psi_{(\omega,A)},\psi_{(\omega,A)}\rangle=-\frac{1}{2}\frac{d}{d\omega}\int_0^{L_0}\psi_{(\omega,A)}^2(x)dx,$$
one has that $(\ref{weinsteincond})$ occurs at the point $(\omega_0,A_0)\in\mathcal{O}$ if, and only if
\begin{equation}\label{calc1}
\frac{1}{2}\frac{d}{d\omega}\int_0^{L_0}\psi_{(\omega,A)}^2(x)dx\Big|_{(\omega,A)=(\omega_0,A_0)}=F_{\omega}(\psi)>0.
\end{equation} Thus, by using that $M_{\omega}(\psi)>0$ and $M(\psi)>0$, we obtain that $\langle\mathcal{L}^{-1}\psi_{(\omega,A)},\psi_{(\omega,A)}\rangle<0$ for all
$(\omega,A)\in\mathcal{O}$.\\

Lemma $\ref{weinstein}$ jointly with $(\ref{calc1})$ will be useful to establish next result.
\begin{lema}\label{lema2}
Let $\psi_{(\omega,A)}$ be as in Theorem $\ref{teoexist}$. Then,
for $(\omega,A)\in \mathcal{O}$, we have
\begin{enumerate}
\item[(i)]{$\displaystyle\inf\{\langle\mathcal{L}f,f\rangle;
\ ||f||=1,\ \langle f, \psi_{(\omega,A)}\rangle=0\}=0.$}\\
\item[(ii)]{$\displaystyle\inf\{\langle\mathcal{L}f,f\rangle;
\ ||f||=1,\ \langle f, \psi_{(\omega,A)}\rangle=0,\ \langle f,
\psi_{(\omega,A)}\psi_{(\omega,A)}'\rangle=0\}>0.$}
\end{enumerate}
\end{lema}
\begin{proof} The proof follows from similar arguments as in \cite[Lemma 4.2]{natali1}.
\end{proof}

\textit{Proof of the Corollary $\ref{coroest1}$.} Firstly, we estimate term
$\langle\mathcal{L}v,v\rangle$ from below by assuming without loss of generality that $||\psi_{(\omega,A)}||=1$. Let us define
\begin{equation}\label{decomp}v_{\bot}=v-v_{||},\ \ \ \mbox{where}\ \ \ v_{||}=\langle v,\psi_{(\omega,A)}\rangle \psi_{(\omega,A)}.\end{equation}
From $(\ref{compat})$ and $(\ref{decomp})$ one has
$$\langle v_{\bot},\psi_{(\omega,A)}\psi_{(\omega,A)}'\rangle=\langle v,\psi_{(\omega,A)}\psi_{(\omega,A)}'\rangle-\langle v_{||},\psi_{(\omega,A)}\psi_{(\omega,A)}'\rangle
=\langle v,\psi_{(\omega,A)}\rangle\langle \psi_{(\omega,A)}^2,\psi_{(\omega,A)}'\rangle=0.$$
In addition, since $\langle v_{\bot},\psi_{(\omega,A)}\rangle=0$, Lemma $\ref{lema2}$ yields
\begin{equation}\label{propy1}
\langle \mathcal{L}v_{\bot},v_{\bot}\rangle \geq C_1||v_{\bot}||^2,
\end{equation}
for some $C_1>0$.\\
\indent Assuming first that
$||u_0||=||\psi_{(\omega,A)}||=1$. Since $F$ is a conserved quantity, we obtain $||u(t)||^2=1$ for
all $t$. Hence, because \eqref{soltrav11} is invariant by translations,
we obtain $\langle
v,\psi_{(\omega,A)}\rangle\geq-C_2||v||_{X^{m_2}}^4$. Thus, there are
positive constants $C_3$ and $C_4$ such that
\begin{equation}\label{Pperp}
\langle\mathcal{L}v_{\bot},v_{\bot}\rangle\geq
C_3||v||^2-C_4||v||_{X^{\frac{m_2}{2}},\omega}^4,
\end{equation}
where $||f||_{X^{\frac{m_2}{2}},\omega}^2:=\int_0^{L_0}(\mathcal{M}^{1/2}f(x))^2dx+\omega\int_0^{L_0}f(x)^2dx$ is an equivalent norm in $X^{\frac{m_2}{2}}$. Next, from the Cauchy-Schwartz inequality,
\begin{equation}\label{Pperpar}
\langle\mathcal{L}v_{||},v_{\bot}\rangle\geq-C_5||v||_{X^{\frac{m_2}{2}},\omega}^3,
\end{equation}
 for some $C_5>0$. Therefore, \eqref{Pperp} and
 \eqref{Pperpar}, yield
\begin{equation}\label{est2}
\langle\mathcal{L}v,v\rangle\geq
C_6||v||_{X^{\frac{m_2}{2}},\omega}^2-C_7||v||_{X^{\frac{m_2}{2}},\omega}^3-C_8||v||_{X^{\frac{m_2}{2}},\omega}^4,
\end{equation}
where $C_i>0$, $i=6,7,8$. Finally, collecting results in $(\ref{deltaR})$ and $(\ref{est2})$ we have
\begin{equation}\label{finalest}
\Delta\mathcal{P}_{(\omega,A)}\geq D_1||v||_{X^{\frac{m_2}{2}},\omega}^2-D_2||v||_{X^{\frac{m_2}{2}},\omega}^3-D_3||v||_{X^{\frac{m_2}{2}},\omega}^4,
\end{equation}
for some $D_i>0$, $i=1,2,3$. The remainder of the proof can be established by using standard arguments. For details, we refer the reader to see \cite{bona1} (see also \cite{ABS} and \cite{weinstein1}). This argument proves that the orbit generated by $\psi_{(\omega,A)}(x-ct)$ is stable
relative to small perturbations which preserves the $L_{per}^2-$norm
of $\psi_{(\omega,A)}$. The general case (that for $\|u_0\|\neq
\|\psi_{(\omega,A)}\|$) follows from the continuous dependence of
the function $\psi_{(\omega,A)}$ with respect to the parameters $(\omega,A)$ jointly with the triangle inequality.
\begin{flushright}
${\square}$
\end{flushright}

\indent We can summarize the results obtained in this section with the following theorem:
\begin{teo}\label{teoestgeral}
Suppose that assumption $(H)$ occurs. The periodic wave $\psi$ is orbitally stable in the sense of the Definition $\ref{defi1}$ provided that $\frac{M(\psi)}{L_0}>\omega_0>0$.
\end{teo}
\begin{proof} The proof of this result follows immediately from Corollary $\ref{coro2}$ and Corollary $\ref{coroest1}$.
\end{proof}
\section{An Application} This section is devoted to apply the arguments in Section 2 to conclude the orbital stability of periodic waves for the Kawahara equation $(\ref{equakawa1})$. In  reference \cite{ncp}, the authors have constructed a smooth curve $\omega\in I\mapsto\psi_{\omega}\in H_{per}^n([0,L_0])$, $n\in\mathbb{N}$, of $L_0-$periodic waves and proving the orbital stability for specific values of $\omega\in I$ by using the arguments in \cite{andrade}. The method established in \cite{andrade} was an adaptation for the periodic case of the classical theory established in \cite{grillakis1}. In our present approach, we prove the orbital stability without assuming the restrictions on the wave speed $\omega$.\\
\indent Indeed, let us consider the ansatz (see \cite{parkes})
\begin{eqnarray}\label{sol}
\psi(x) = a &+& b\left(\mbox{dn}^2\left(\frac{2K}{L}x,k\right)-\frac{E}{K}\right)   \nonumber \\
&+& d\left(\mbox{dn}^4\left(\frac{2K}{L}x,k\right)-(2-k^2)\frac{2E}{3K}+\frac{1-k^2}{3}\right). \label{sol1}
\end{eqnarray}
Substituting this form into the equation
\begin{equation}\label{eq2}
\psi''''-\psi''+\omega\psi-\frac{1}{2}\psi^2+A=0
\end{equation}
 one has explicit periodic solutions provided that
\begin{eqnarray}\label{a}
a=\frac{1}{507L^4}((-k^4+k^2+1)302848K^4+14560L^2K^2(k^2-2) \nonumber \\ +43680L^2EK+L^4(-31+507\omega)),
\end{eqnarray}
\begin{equation}\label{bd}b=\frac{1120}{13L^4}((208k^2-416)K^2+L^2)K^2 \quad \mbox{and} \quad d=\frac{26880K^4}{L^4}.\end{equation}

Furthermore, $A$ is a complicated function which depends smoothly on the triple $(k,L,\omega)$ and it may be expressed by
\begin{equation}\label{valA}
A=f_1(k,L)+C\omega^2,
\end{equation}
where $C\in\mathbb{R}$.\\
\indent Moreover, we also need to consider a pair $(k,L)$ which solves the following (implicit) nonlinear equation
\be\label{kL}
\frac{89989120}{31}(k^2-2)\left(k^2-\frac{1}{2}\right)(k^2+1)K^6-\frac{908544}{31}L^2(k^4-k^2+1)K^4+L^6=0.
\ee
\indent A standard application of the implicit function theorem gives us the existence of two open intervals $I\subset(0,+\infty)$ and $J\subset(0,1)$ such that the function $k\in J\mapsto L(k)\in I$ is smooth. Therefore, for a fixed value of the modulus $k_0\in(0,1)$ one has a unique value $L_0>0$ such that $\psi$ is a smooth $L_0-$periodic solution related to the equation $(\ref{eq2})$ as required in the first part of assumption $(H)$ (important to mention that $\omega$ is a free parameter which does not depend on the pair $(k,L)$).\\
\indent With this arguments in hands, we need to establish the spectral property associated with the linearized operator
\begin{equation}\label{operakawa}
\mathcal{L}_{(\omega_0,A_0)}=\partial_x^4-\partial_x^2+\omega_0-\psi,\ \ \ \ \ \ \ \ \omega_0>0.
\end{equation}

\begin{prop}\label{propspec}Consider $L_0>0$ satisfying $(\ref{kL})$ and let $\omega_0>0$ be arbitrary but fixed. The operator $\mathcal{L}_{(\omega_0,A_0)}$ in (\ref{operakawa}) possesses exactly a unique negative
eigenvalue which is simple, and zero is a simple eigenvalue with
eigenfunction $\frac{d}{dx}\psi$.
\end{prop}

\begin{proof}
 We prove the result by using Theorem 4.1 in \cite{AN}. First of all, we use the Galilean invariance associated to $(\ref{eq2})$ in order to prove that the spectral property in $(H)$ will be the same for all values of $\omega\in\mathbb{R}$ (the value of the integration constant $A$ is irrelevant in our spectral analysis). Therefore, it suffices to prove the result for a specific value of $\omega_0$. In fact, let $\alpha_0\in\mathbb{R}$ be arbitrary but fixed. By defining $\tilde{\psi}=\alpha_0+\psi$, where $\psi$ is solution of \eqref{eq2}, it follows that
$$(\omega_0+\alpha_0)\tilde{\psi}-\frac{1}{2}\tilde{\psi}^{2}-\tilde{\psi}''+\tilde{\psi}''''+\tilde{A_0}=0,$$
where $\tilde{A_0}=A_0-\omega_0\alpha_0-\frac{\alpha_0^2}{2}$. Therefore, $\tilde{\psi}$ solves a similar equation as in $(\ref{eq2})$ with wave speed $\omega_0+\alpha_0$. Thus, we obtain
$$\mathcal{L}_{(\omega_0,A_0)}=\frac{\partial^4}{\partial x^4}-\frac{\partial^2}{\partial x^2}+\omega_0-\psi=\frac{\partial^4}{\partial x^4}-\frac{\partial^2}{\partial x^2}+\omega_0+\alpha_0-\tilde{\psi}=:\mathcal{L}_{(\omega_0+\alpha_0,\tilde{A_0})}.$$
Last equality gives us the desired result.\\
\indent Next, we can write the solution (\ref{sol}) as a Fourier series of the form (see \cite{ayse})
	
$$
\phi(x) = a_0+\gamma\sum_{n=1}^{\infty}n\mbox{csch}\left(\frac{n\pi K_0'}{K_0}\right)\cos\left(\frac{2\pi n}{L_0}x\right), \nonumber
$$	
where $\gamma_0:=\left(\frac{b_0\pi^2}{K_0^2}+\frac{d_0\pi^2}{k^2K_0^2}\left(\frac{4-2k_0^2}{3}+\frac{n^2\pi^2}{6K_0}\right)\right)$, $K_0=K(k_0)$ and $K'(k_0)=K(\sqrt{1-k_0^2})$. In addition, according with $(\ref{a})$ and $(\ref{bd})$ we can write $a_0=a(k_0,L_0,\omega_0)$, $b_0=b(k_0,L_0)$, $d_0=d(k_0,L_0)$. Therefore, the Fourier coefficients are
\begin{equation}\label{coefficients}
\hat{\psi}(n)=\left \{
\begin{array}{cc}
a_0, & n=0 \\
\sigma(n), & n\neq0 \\
\end{array}
\right.
\end{equation}
where
$\sigma(n)=\frac{\gamma_0}{2} n\mbox{csch}\left(\frac{n\pi K_0'}{K_0}\right)$.

 By defining $g(x)=\frac{\gamma_0}{2} x\mbox{csch}\left(\frac{x\pi K_0'}{K_0}\right)$, $x\in\mathbb{R}$, we see that $g$ is a smooth logarithmically concave function. Thus, from Lemma 4.1 in \cite{AN} one has that $g$ belongs to the $PF(2)-$continuous class and thus, since $a_0>g(0)$, for all $\omega_0>0$ large enough, we can redefine the $PF(2)$-continuous function $g$
by a differentiable function $s:\mathbb{R}\rightarrow\mathbb{R}$
such that $s(0)= a_0$,
$s(x)=g(x)$ in $(-\infty,-1]\cup[1,+\infty)$ such that $s\in PF(2)$ in the continuous case. Letting $s(n)=\hat{\psi}(n)$, $n\in\mathbb{Z}$, one has that $(\widehat{\psi}(n))_{n\in\mathbb{Z}}$ belongs to $PF(2)$, for all $\omega_0>0$ large enough. Therefore, Theorem 4.1 in \cite{AN} gives us the spectral properties required in assumption $(H)$.

\end{proof}

\indent Next, we need to analyze the difference $\frac{M(\psi)}{L_0}-\omega_0$ to conclude the orbital stability of periodic waves for the model $(\ref{equakawa1})$. Indeed, since $M(\psi)=a_0L$, we can deduce from $(\ref{a})$ that $a_0-\omega_0$ just depends on the pair $(k_0,L_0)\in J\times I$. Therefore, one has
$$
\begin{array}{llll}\frac{M(\psi)}{L_0}-\omega_0&=&a_0-\omega_0\\\\
&=&\displaystyle \frac{302848(-k_0^4+k_0^2+1)K_0^4+14560L^2K_0^2(k_0^2-2)+43680L^2E_0K_0-31L^4}{507L^4}\\\\
&=&\displaystyle\frac{1}{507L^4}p(k_0,L_0^2),
\end{array}
$$
where $E_0=E(k_0)$. By taking $L_1=L_0^2$ we can rewrite function $p(k_0,L_0^2)$ as $p(k_0,L_1)$. This change of variables can be used to simplify the implicit relation in $k_0$ and $L_0$ in \eqref{kL} as
\be\label{kL1}
\frac{89989120}{31}(k_0^2-2)\left(k_0^2-\frac{1}{2}\right)(k_0^2+1)K_0^6-\frac{908544}{31}L_1(k_0^4-k_0^2+1)K_0^4+L_1^3=0.
\ee

Using \textit{Maple 16}, we can solve algebraically the equation in \eqref{kL1} in terms of the modulus in order of obtaining the positive function
$$ L_1(k)=\frac{104}{31}\frac{r(k_0)}{q(k_0)},
$$
where
$r(k_0)$ and $q(k_0)$ are complicated expressions containing several powers of $k_0$. Since $\frac{M(\psi)}{L_0}-\omega_0=\frac{1}{507L^4}p(k_0,L_1)$,
we can plot the graph of $p(k_0,L_1)$ in order to understand its behaviour in terms of the modulus. The figure below shows that there are values of the pair $k_0$ such that the difference $p(k_0,L_1)$ is positive as required in our stability approach.
\begin{figure}[h!]
\includegraphics[scale=0.30]{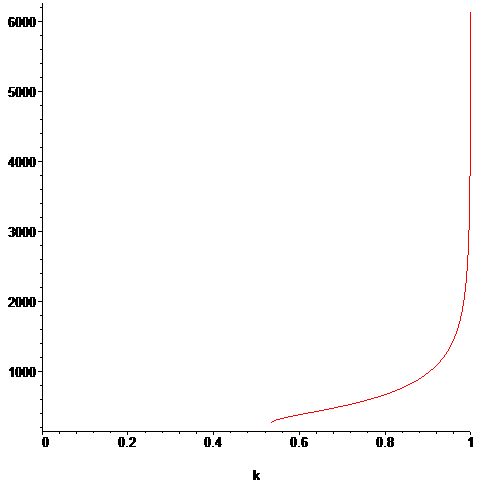}\quad \quad \quad \quad \includegraphics[scale=0.30]{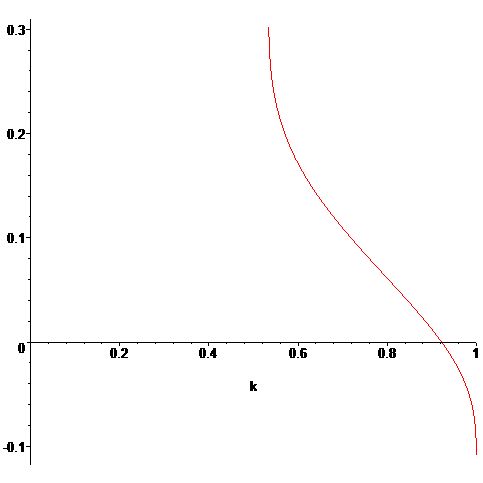}
\caption{Left: The graph of the function $L_1(k_0)$. Right: The graph of $p(k_0,L_1)$.}
\label{figure4}
\end{figure}
Important to mention that our results are agreeing with those ones in \cite{ncp} since to conclude the stability in refereed paper, it makes necessary to analyse the behaviour of the difference $\frac{M(\psi)}{L_0}-\omega_0$. The main problem in \cite{ncp} is that we need, in order to use an adaptation of the arguments in \cite{grillakis1}, to consider small values of $\omega_0>0$ to determine a positiveness of a certain quantity. This fact is not necessary and our stability result becomes more complete. Thus, collecting all results above we are enable to enunciate the following result.

\begin{teo}
Consider $L_0>0$ satisfying $(\ref{kL})$ and let $\omega_0>0$ be arbitrary but fixed. The traveling wave $\psi(x-\omega_0t)$ in (\ref{sol1}) is orbitally stable in $H_{per}^2([0,L_0])$ by the periodic flow of the equation (\ref{equakawa1}) provided that $\frac{M(\psi)}{L_0}>\omega_0$.
\end{teo}

\begin{obs}
Global solutions in the energy space $H_{per}^2([0,L_0])$ as well as existence of convenient conserved quantities as in $(\ref{conser1})$, $(\ref{conser2})$ and $(\ref{conser2})$ with $\mathcal{M}=\partial_x^4-\partial_x^2$ associated with the equation $(\ref{equakawa1})$ can be found in reference \cite{tkato}.

\end{obs}

\section*{Acknowledgement}

F. N. is partially supported by CNPq/Brazil.


\begin{thebibliography}{99}


\bibitem{andrade} T.P. Andrade and A. Pastor, \textit{Orbital stability of periodic traveling-wave solutions for the BBM equation with fractional nonlinear term}, Phys. D, 317 (2016), p. 43-58. 






\bibitem{AN} J. Angulo and F. Natali, \textit{Positivity properties of the Fourier transform and the stability of
periodic travelling-wave solutions}, SIAM J. Math. Anal.,
40 (2008), pp. 1123-1151.



\bibitem{ABS} J. Angulo, J. L. Bona and M. Scialom, \textit{Stability of cnoidal waves},
Adv. Diff. Equat., 11 (2006), pp. 1321-1374.


\bibitem{be}  T. B. Benjamin, \textit{The stability of solitary waves,}  Proc. Roy. Soc. (London) Ser. A 328 (1972), pp. 153-183.



\bibitem{bona2} J.L. Bona, P.E. Souganidis and W.A. Strauss, \textit{Stability and instability of
solitary waves of Korteweg-de Vries type}, Proc. Roy. Soc. Lond.
Ser. A 411 (1987), pp. 395-412.

\bibitem{bona1} J. L. Bona, \textit{On the stability theory of solitary
waves,} Proc. R. Soc. Lond. Ser. A, 344 (1975), pp. 363-374.


\bibitem{ncp} F. Crist\'ofani, F. Natali and T.P. Andrade, \textit{Orbital stability of periodic traveling wave solutions for the Kawahara equation}, J. Math. Phys., 58 (2017), 051504.





\bibitem{FL} R.L. Frank and E. Lenzmann, \textit{Uniqueness of non-linear ground states for
fractional Laplacians in $\mathbb{R}$}, Acta Math., 210 (2013), pp. 261–318.

\bibitem{grillakis1} M. Grillakis, J. Shatah and W. Strauss, \textit{Stability theory of solitary waves in the presence of symmetry
I.} J. Funct. Anal., 74 (1987), pp. 160-197.

\bibitem{Haragus} M. H$\check{a}$r$\check{a}$gu\c{s} and  T. Kapitula, \textit{On the spectra of periodic waves for infinite-dimensional Hamiltonian systems},  Phys. D, 237 (2008), pp. 2649-2671.

    \bibitem{haragus1} M. H$\check{a}$r$\check{a}$gu\c{s}, E. Lombardi and A. Scheel,  \textit{Spectral stability of wave trains in the Kawahara equation},
    J. Math. Fluid Mech., 8 (2006), pp. 482-509.

    \bibitem{hur} V.M. Hur and M. Johnson, \textit{Stability of periodic traveling waves
for nonlinear dispersive equations}, SIAM J. Math. Anal., 47, pp. 3528–3554.



\bibitem{johnson} M. Johnson, \textit{Nonlinear stability
of periodic traveling wave solutions of the generalized Korteweg-de
Vries equation}, SIAM J. Math. Anal., 41 (2009), pp. 1921-1947.


\bibitem{tkato} T. Kato, {\it Low regularity well-posedness for the periodic Kawahara equation}, Diff. Int. Equat., 25 (2012), pp. 1011-1036.

\bibitem{kato1} T. Kato, \textit{Perturbation theory for linear Operators},
Springer, Berlin, (1976).



\bibitem{ayse} A. Kiper, \textit{Fourier series coefficients for powers of
the Jacobian Elliptic Functions}, Math. Comput., 43 (1984), pp. 247-259.


\bibitem{natali1} F. Natali and A. Neves, \textit{Orbital stability of solitary waves}. IMA J. Appl. Math.,  79 (2014), pp. 1161-1179.







\bibitem{parkes} E.J. Parkes, B.R. Duffy and P.C. Abbot, {\it The Jacobi elliptic-function method for finding periodic-wave solutions to nonlinear evolution equations}. Phys. Lett. A, 295 (2002), pp. 280-286.






\bibitem{weinstein1} M.I. Weinstein, \textit{Modulation Stability of Ground States of Nonlinear Schr\"{o}dinger
Equations}. SIAM J. Math., 16 (1985), pp. 472-490.

\bibitem{weinstein2} M.I. Weinstein, \textit{Liapunov stability of ground states of nonlinear dispersive equations}.  Comm. Pure Appl. Math., 39 (1986), pp. 51-68.

\end{thebibliography}
\end{document}